\documentclass[11pt,a4paper]{amsart}

\usepackage{amssymb}
\usepackage[authoryear]{natbib}
\usepackage[final]{showkeys}
\usepackage{microtype}
\usepackage{color}
\usepackage{graphicx}
\usepackage[hmargin=3cm,vmargin={3.5cm,4cm}]{geometry}
\newtheorem{te}{Theorem}[section]

\newtheorem{os}[te]{Remark}

\newtheorem{lem}[te]{Lemma}

\numberwithin{equation}{section}

\allowdisplaybreaks

\begin{document}

    \title{Random flights governed by Klein-Gordon-type partial differential equations}

    \author{Roberto Garra$^1$}
    \address{${}^1$Dipartimento di Scienze di Base e Applicate per l'Ingegneria, ``Sapienza'' Universit\`a di Roma.}

    \author{Enzo Orsingher$^2$}
    \address{${}^2$Dipartimento di Scienze Statistiche, ``Sapienza'' Universit\`a di Roma.}



    \date{\today}

    \begin{abstract}
     In this paper we study random flights in $\mathbb{R}^{\mathfrak{d}}$ with displacements possessing Dirichlet 
     distributions of two different types and uniformly oriented. The randomization of the number
     of displacements has the form of a generalized Poisson process whose parameters depend on the
     dimension $\mathfrak{d}$. We prove that the distributions of the point $\mathbf{X}(t)$ and $\mathbf{Y}(t)$, $t\geq 0$, performing
     the random flights (with the first and second form of Dirichlet intertimes) are related to 
     Klein-Gordon-type p.d.e.'s. Our analysis is based on McBride theory of integer powers of hyper-Bessel operators.
     A special attention is devoted to the three-dimensional case where we are able to obtain the explicit form 
    of the equations governing the law of $\mathbf{X}(t)$ and $\mathbf{Y}(t)$. In particular we show that
     that the distribution of $\mathbf{Y}(t)$ satisfies a telegraph-type equation with time-varying coefficients.
	
    \textbf{Mathematics~Subject~Classification~(2010):} 60G60
    \end{abstract}

    \maketitle

    \section{Introduction}

    Random flights in $\mathbb{R}^{\mathfrak{d}}$ have been introduced since the beginning of the Twenteeth Century
    by Pearson, Kluyver and Rayleigh. In their works uniformly oriented random displacements of fixed
    length
    are considered. Random flights where changes of direction are spaced by a homogeneous Poisson process
    with displacements uniformly distributed have been considered in $\mathbb{R}^{\mathfrak{d}}$ from different
    viewpoints by Stadje (in $\mathbb{R}^2$ \citep{sta} and $\mathbb{R}^3$ \citep{sta1}), \citet{ale2}, \citet{fra} and \citet{garcia}.
    The case of random flights with Dirichlet distributed displacements (also uniformly distributed) was considered and investigated
    by \citet{Lec} and \citet{ale}. Possible applications of random flights are suggested by the scattering of
    light rays in inhomogeneous media. Recent applications to the analysis of photon propagation in the
    Cosmic Microwave Background (CMB) radiation have been discussed in \citet{rei}.
    Furthermore, \citet{luca} have shown that the probability law
    of planar random motions discussed in \citet{kol} coincides with the
    explicit form of the van Hove function for the run-and-tumble
    model in two dimensions. This work gives and interesting and
    strong link between explicit solutions of the Lorentz model of
    electron conduction and the probability theory of random
    flights.\\
    Displacements have random orientation defined by the angles $(\theta_1, \theta_2, \dots, \theta_{\mathfrak{d}-2}, \phi)$
    with density
    \begin{equation}\label{1}
     g(\theta_1, \dots, \theta_{\mathfrak{d}-2}, \phi)=\frac{\Gamma\left(\frac{\mathfrak{d}}{2}\right)}{2\pi^{\mathfrak{d}/2}}\sin^{\mathfrak{d}-2}\theta_1
    \sin^{\mathfrak{d}-3}\theta_2 \dots \sin \theta_{\mathfrak{d}-2},
    \end{equation}
    with $0\leq\theta_j\leq \pi$, $0\leq \phi\leq 2\pi$.\\
    The length of displacements $\tau_j$, $j=1, \dots, k$ between successive changes of direction occurring at times $t_j$, $1\leq j\leq k$,
    with $\tau_j = t_j-t_{j-1}$ has distribution

   \begin{equation}
   f_1(\tau_1, \dots, \tau_k)=\frac{\Gamma((k+1)(\mathfrak{d}-1))}{[\Gamma(\mathfrak{d}-1)]^{k+1}}
     \frac{1}{t^{(k+1)(\mathfrak{d}-1)-1}}\prod_{j=1}^{k+1}\tau_j^{\mathfrak{d}-2},
   \end{equation}
     where $0<\tau_j< t-\sum_{n=0}^{j-1}\tau_n$, $1\leq j\leq k$, $\tau_{k+1}=t-\sum_{j=1}^{k}\tau_j$,
     which is a Dirichlet distribution
     with parameters $(\mathfrak{d}-1, \dots, \mathfrak{d}-1)$, with $\mathfrak{d}\geq 2$. The probability density of the vector
     $\mathbf{X}_{\mathfrak{d}}(t)=(X_1(t), \dots, X_{\mathfrak{d}}(t))$ representing the position of the moving particle at time $t$ after
     $k$ changes of direction reads
     \begin{equation}\label{a20}
     p_{\mathbf{X}_{\mathfrak{d}}}(\mathbf{x}_{\mathfrak{d}}, t ; k)=\frac{\Gamma(\frac{k+1}{2}(\mathfrak{d}-1)+\frac{1}{2})}{\Gamma(\frac{k}{2}
     (\mathfrak{d}-1))}
     \frac{(c^2t^2-\|\mathbf{x}_{\mathfrak{d}}\|^2)^{\frac{k}{2}(\mathfrak{d}-1)-1}}{\pi^{\mathfrak{d}/2}(ct)^{(k+1)(\mathfrak{d}-1)-1}},
     \end{equation}
     with $\|\mathbf{x}_{\mathfrak{d}}\|<ct$, $\mathfrak{d}\geq 2$ (see Theorem 2 of \citet{ale}). For $\mathfrak{d}=2$, formula \eqref{a20} reduces to
    \begin{equation}
      p_{\mathbf{X}_{\mathit{2}}}(\mathbf{x}_{\mathit{2}}, t ; k)=\frac{k}{2\pi(ct)^k}
     (c^2t^2-\|\mathbf{x}_{\mathit{2}}\|^2)^{\frac{k}{2}-1}, \quad k\geq 1,
    \end{equation}
    and was firstly obtained by \citet{sta} in relation to finite-velocity planar random motions
    (see also \citet{kol}).\\
    In \citet{ale} and \citet{Lec1} is shown that, if displacements $(\tau_1, \dots, \tau_k)$ have joint Dirichlet
    distribution with parameters $(\frac{\mathfrak{d}}{2}-1, \dots,\frac{\mathfrak{d}}{2}-1)$, that is,
    \begin{equation}
     f_2(\tau_1, \dots, \tau_k)=\frac{\Gamma((k+1)(\frac{\mathfrak{d}}{2}-1))}{\Gamma(\frac{\mathfrak{d}}{2}-1)^{k+1}}
     \frac{1}{t^{(k+1)(\frac{\mathfrak{d}}{2}-1)-1}}\prod_{j=1}^{k+1}\tau_j^{\frac{\mathfrak{d}}{2}-2},
     \end{equation}
     the density of the vector
     $\mathbf{Y}_{\mathfrak{d}}(t)=(Y_1(t), \dots, Y_{\mathfrak{d}}(t))$ becomes
     \begin{equation}\label{a00}
     p_{\mathbf{Y}_{\mathfrak{d}}}(\mathbf{y}_{\mathfrak{d}}, t ; k)=\frac{\Gamma((k+1)(\frac{\mathfrak{d}}{2}-1)+1)}{\Gamma(k(\frac{\mathfrak{d}}{2}-1))}
     \frac{(c^2t^2-\|\mathbf{y}_{\mathfrak{d}}\|^2)^{k(\frac{\mathfrak{d}}{2}-1)-1}}{\pi^{\mathfrak{d}/2}(ct)^{2(k+1)(\frac{\mathfrak{d}}{2}-1)}},
     \end{equation}
     with $\|\mathbf{y}_{\mathfrak{d}}\|<ct$, $k\geq1$, $\mathfrak{d}\geq 3$.\\
     For $\mathfrak{d}=4$ we extract from \eqref{a00}
      \begin{equation}
     p_{\mathbf{Y}_{\mathit{4}}}(\mathbf{y}_{\mathit{4}}, t ; k)=\frac{k(k+1)}{\pi^2 (ct)^{2k+2}}
     (c^2t^2-\|\mathbf{y}_{\mathit{4}}\|^2)^{k-1},
     \end{equation}
     which coincides with formula (1.5) of \citet{ale2}.\\
     In order to obtain the unconditional distributions of $\mathbf{X}_{\mathfrak{d}}(t)$ and $\mathbf{Y}_{\mathfrak{d}}(t)$, we
     here teke into account a randomization different from that applied in \citet{ale}. \\
     In the case of $\mathbf{X}_{\mathfrak{d}}(t)$, $t\geq 0$, we consider the following distribution
     \begin{equation}\label{b20}
     P\{\mathfrak{N}_{\mathfrak{d}}(t)=k\}=\frac{1}{E_{\mathfrak{d}-1,\mathfrak{d}-1}\left(\left(\lambda t\right)^{\mathfrak{d}-1}\right)}
     \frac{(\lambda t)^{k(\mathfrak{d}-1)}}{\Gamma((k+1)(\mathfrak{d}-1))},
     \end{equation}
     with $\lambda > 0$, $\mathfrak{d}\geq 2$, $k=0, 1,\dots,$ and
     \begin{equation}
     E_{\alpha, \beta}(x)=\sum_{k=0}^{\infty}\frac{x^k}{\Gamma(\alpha k+ \beta)},
     \end{equation}
     represents the generalized Mittag-Leffler function (see for example \citet{ky}).
     We observe that \eqref{b20} includes as a special case for $\mathfrak{d}=2$ the distribution of an homogeneous Poisson process.
     It can be regarded as a generalization of the Poisson process in the sense that will be discussed in section 3.\\
     In the case of $\mathbf{Y}_{\mathfrak{d}}(t)$, $t\geq 0$, the randomization is performed by means of the process
     with the following distribution
     \begin{align}\label{b0}
     P\{\mathcal{N}_{\mathfrak{d}}(t)=k\}=\frac{1}{E_{\mathfrak{d}-2,\mathfrak{d}-1}
     \left((\lambda t)^{\mathfrak{d}-2}\right)}
     \frac{(\lambda t)^{k(\mathfrak{d}-2)}}{\Gamma((\mathfrak{d}-2)k+\mathfrak{d}-1)},
     \end{align}
     with $\lambda > 0$, $\mathfrak{d}\geq 3$, $k=0, 1,\dots,$.\\
     By randomizing the distribution \eqref{a20} with \eqref{b20} and \eqref{a00} with \eqref{b0} we obtain
     the unconditional distribution of $\mathbf{X}_{\mathfrak{d}}(t)$ and $\mathbf{Y}_{\mathfrak{d}}(t)$, for $\mathfrak{d}\geq 2$
     in the first case and $\mathfrak{d}\geq 3$ in the second one.
     For some values of the dimension $\mathfrak{d}$ these distributions have an attractive form. For example for $\mathfrak{d}=3$, that is for the most
     interesting case for the applications, we have that
     \begin{align}\label{tred}
      \frac{P\{\mathbf{X}_{3}(t)\in d\mathbf{x}_3\}}{\prod_{j=1}^{3} dx_j}=
      \left(\frac{\lambda}{2c}\right)^2\frac{1}{\pi \sinh(\lambda t)}
     \frac{I_1\left(\frac{\lambda}{c}\sqrt{c^2t^2-\|\mathbf{x}_3\|^2}\right)}{\sqrt{c^2t^2-\|\mathbf{x}_3\|^2}},
     \end{align}
     and
     \begin{align}\label{tredo}
      \frac{P\{\mathbf{Y}_{3}(t)\in d\mathbf{y}_3\}}{\prod_{j=1}^{3} dy_j}&=
      \frac{\lambda}{2c}\frac{1}{\pi (e^{\lambda t}-1)}
     \sum_{k=0}^{\infty}\left(\frac{\lambda}{2c}\right)^{k+1}\frac{\left(\sqrt{c^2t^2-\|\mathbf{y}_3\|^2}\right)^{k-1}}{\Gamma(\frac{k+1}{2})\Gamma(\frac{k+3}{2})}\\
     \nonumber &=\left(\frac{\lambda}{2c}\right)^2\frac{1}{\pi(e^{\lambda t}-1)}
      \frac{E_{\frac{1}{2},\frac{1}{2},\frac{1}{2},\frac{3}{2}}\left(\frac{\lambda}{2c}\sqrt{c^2t^2-\|\mathbf{y}_3\|^2}\right)}{\sqrt{c^2t^2-\|\mathbf{y}_3\|^2}},
     \end{align}
     for $\mathbf{x}_3$ and $\mathbf{y}_3$ belonging to $S^3_{ct}:=\{\mathbf{x}\in \mathbb{R}^3: \|\mathbf{x}_3\|^2\leq c^2t^2\}$.\\
     In \eqref{tred}
     \begin{equation}
     I_1(x)=\sum_{k=0}^{\infty}\left(\frac{x}{2}\right)^{2k+1}\frac{1}{k!(k+1)!},
     \end{equation}
     while
     \begin{equation}
     E_{\frac{1}{2},\frac{1}{2},\frac{1}{2},\frac{3}{2}}(x)=\sum_{k=0}^{\infty}\frac{x^k}{\Gamma(\frac{k+1}{2})\Gamma(\frac{k+3}{2})},
     \end{equation}
     is a multi-index Mittag-Leffler function. While the distribution \eqref{tred} is finite near $\partial S^3_{ct}$, the density
     \eqref{tredo} converges to $+\infty$ as $\|\mathbf{y}_3\|\rightarrow ct$.\\
     We present the partial differential equations governing the space-dependent component of the distributions
     $P\{\mathbf{X}_{\mathfrak{d}}(t)\in d\mathbf{x}_{\mathfrak{d}}\}$ and $P\{\mathbf{Y}_{\mathfrak{d}}(t)\in d\mathbf{x}_{\mathfrak{d}}\}$.
     For the case $\mathfrak{d}=3$, we prove that $P\{\mathbf{Y}_{3}(t)\in d\mathbf{y}_3\}$ satisfies a telegraph-type equation
     with time-dependent coefficients (see \eqref{varte}), while $P\{\mathbf{X}_{3}(t)\in d\mathbf{y}_3\}$ satisfies an higher order equation
     involving the second-order D'Alembert operator.\\
    For a random flight in $\mathbb{R}^3$ where changes of direction occur only at even-valued
     Poisson times, the explicit distribution of the current position was discussed in \citet{ale} (see formula \eqref{p3} below).
     Here we prove that its probability law satisfies the three dimensional telegraph equation
     \begin{equation}
     \frac{\partial^2 u}{\partial t^2}+2\lambda \frac{\partial u}{\partial t}-c^2\Delta u=0.
     \end{equation}

    \section{Preliminaries about McBride theory}

    For our analysis the integer power of hyper-Bessel operators play a crucial role.
    For this reason we devote this section to the presentation of the basic facts of McBride theory introduced in (1975) and then extended
    in a series of successive papers.
        \label{prel}
        Our starting point is the generalized hyper-Bessel operator, considered in \citet{mc2},
        \begin{equation}
            \label{L}
            L=x^{a_1}Dx^{a_2}\dots x^{a_n}Dx^{a_{n+1}},
        \end{equation}
        where $n$ is an integer number, $a_1,\dots, a_{n+1}$ are complex numbers and $D=d/dx$.
        Hereafter we assume that the coefficients $a_j$, $j=1,\dots, n+1$ are real numbers.
        The operator $L$ generalizes the classical $n$-th order hyper-Bessel operator
        \begin{align*}
            L_{B_n}=x^{-n}\underbrace{x\frac{d}{dx}x\frac{d}{dx}\dots x\frac{d}{dx}}_{\text{$n$ times}}.
        \end{align*}
        The operator $L$ defined in \eqref{L} acts on the functional space
        \begin{equation}
            F_{p,\mu}=\{f \colon x^{-\mu}f(x)\in F_p\},
        \end{equation}
        where
        \begin{equation}
            F_p=\{f\in C^{\infty} \colon x^k d^k f/dx^k \in L^p, k=0,1,\dots\},
        \end{equation}
        for $1 \leq p < \infty$ and for any complex number $\mu$ \citep[see][for details]{mc2, mc}.
        The following lemma gives an alternative representation of the operator $L$.
        \begin{lem}
            \label{duepuntouno}
            The operator $L$ in \eqref{L} can be written as
            \begin{equation}
                \label{Lo}
                L f= m^{n}x^{a-n}\prod_{k=1}^n x^{m-m b_k}D_m x^{m b_k} f,
            \end{equation}
            where
            \begin{align*}
                D_m := \frac{d}{dx^m}=m^{-1}x^{1-m}\frac{d}{dx}.
            \end{align*}
            The constants appearing in \eqref{Lo} are defined as
            \begin{align*}
                a=\sum_{k=1}^{n+1}a_k, \qquad m= |a-n|,
                \qquad b_k= \frac{1}{m}\left(\sum_{i=k+1}^{n+1}a_i+k-n\right), \quad k=1, \dots, n.
            \end{align*}
        \end{lem}
        For the proof see lemma 3.1, page 525 of \citet{mc}.

        In the analysis of the integer power (as well as the fractional power) of the operator $L$,
        a key role is played by $D_m$
        appearing in \eqref{Lo}.

        \begin{lem}\label{pot}
            Let $r$ be a positive integer, $a<n$, $f\in F_{p,\mu}$ and
            \begin{align*}
                b_k\in A_{p,\mu,m}:=\{\eta \in \mathbb{C}
                \colon \Re(m\eta+\mu)+m\neq 1/p-ml, \: l=0, 1, 2,\dots\}, \qquad k=1,\dots, n.
            \end{align*}
            Then
            \begin{equation}
                L^r f= m^{nr}x^{-mr}\prod_{k=1}^n I^{b_k,-r}_m f,
            \end{equation}
            where, for $\alpha >0$ and $\Re(m\eta+\mu)+m > 1/p$
            \begin{equation}
                \label{mc1-2}
                \left(I_m^{\eta, \alpha}f\right)(x)=
                \frac{x^{-m\eta-m\alpha}}{\Gamma(\alpha)}\int_0^x(x^m-u^m)^{\alpha-1}u^{m\eta}f(u)\, d(u^m),
            \end{equation}
            and for $\alpha\leq 0$
            \begin{equation}
                \label{mc2}
                \left(I_m^{\eta, \alpha}f\right)(x)=(\eta+\alpha+1)I_m^{\eta, \alpha+1}f+\frac{1}{m} I_m^{\eta, \alpha+1}
                \left(x\frac{d}{dx}f\right).
            \end{equation}
        \end{lem}
        For the proof consult \citet{mc}, page 525.
        By using similar arguments, it is also possible to give the fractional generalization
        $L^{\alpha}$ of $L$. A useful result that will be used in the
        following section is given by the Lemma
            \begin{lem}
            \label{brunello}
            Let be $\eta+\frac{\beta}{m}+1 >0$, $m\in \mathbb{N}$, we have
            that
            \begin{equation}
                I_m^{\eta,\alpha}x^{\beta}=\frac{\Gamma\left(\eta+\frac{\beta}{m}+1\right)}
                {\Gamma\left(\alpha+\eta+1+\frac{\beta}{m}\right)}x^{\beta}.
            \end{equation}
        \end{lem}

    \section{Random flights governed by higher order partial differential equations}

    \subsection{The first case}

     We first consider the random flights treated in \citet{ale} and strictly related
     to the finite velocity planar random motions studied in
     \citet{kol}. The random flights in $\mathbb{R}^{\mathfrak{d}}$ consist of the triple $(\mathbf{\theta}, \mathbf{\tau},\mathfrak{N}_{\mathfrak{d}}(t))$,
     where the orientation $\mathbf{\theta}$ has distribution \eqref{1}, $\mathbf{\tau}$ represents the displacements
     and $\mathfrak{N}_{\mathfrak{d}}(t)$ gives the number of changes of direction recorded in $(0,t)$.
     The displacements, in the first model, for $\mathfrak{d}\geq 2$ and $\mathfrak{N}_{\mathfrak{d}}(t)=k$, have length $\mathbf{\tau}=
     (\tau_1, \dots, \tau_k)$ with joint distribution
     \begin{equation}\label{prim}
     f_1(\tau_1, \dots, \tau_k)=\frac{\Gamma((k+1)(\mathfrak{d}-1))}{[\Gamma(\mathfrak{d}-1)]^{k+1}}
     \frac{1}{t^{(k+1)(\mathfrak{d}-1)-1}}\prod_{j=1}^{k+1}\tau_j^{\mathfrak{d}-2},
     \end{equation}
     where $0<\tau_j< t-\sum_{n=0}^{j-1}\tau_n$, $1\leq j\leq k$, $\tau_{k+1}=t-\sum_{j=1}^{k}\tau_j$,
     which is a Dirichlet distribution
     with parameters $(\mathfrak{d}-1, \dots, \mathfrak{d}-1)$. In \citet{ale}, it was shown
     (Theorem 2) that the distribution of the moving point
     $\mathbf{X}_{\mathfrak{d}}(t)=(X_1(t), \dots, X_{\mathfrak{d}}(t))$ (with intermediate displacements
     possessing joint distribution \eqref{prim}) is given by
     \begin{equation}\label{a2}
     p_{\mathbf{X}_{\mathfrak{d}}}(\mathbf{x}_\mathfrak{d}, t ; k)=\frac{\Gamma(\frac{k+1}{2}
     (\mathfrak{d}-1)+\frac{1}{2})}{\Gamma(\frac{k}{2}
     (\mathfrak{d}-1))}
     \frac{(c^2t^2-\|\mathbf{x}_{\mathfrak{d}}\|^2)^{\frac{k}{2}(\mathfrak{d}-1)-1}}{\pi^{\mathfrak{d}/2}(ct)^{(k+1)(\mathfrak{d}
     -1)-1}},
     \end{equation}
     with $\|\mathbf{x}_{\mathfrak{d}}\|<ct$, $\mathfrak{d}\geq 2$.\\
     In the case $\mathfrak{d}=2$, the density \eqref{a2} becomes
     \begin{equation}
     p_{\mathbf{X}_2}(\mathbf{x}_2, t ; k)=\frac{k}{2\pi(ct)^k}(c^2t^2-\|\mathbf{x}_2\|^2)^{\frac{k}{2}-1}, \quad \|\mathbf{x}_2\|^2\leq c^2t^2,
     \end{equation}
     and coincides with formula (11) of \citet{kol} for planar random motions.

     The number $\mathfrak{N}_{\mathfrak{d}}(t)$ of changes of direction is represented by an extension
     of the Poisson process, whose parameters depend on the
     dimension $\mathfrak{d}$, and has distribution of the following form
     \begin{align}\label{b2}
    &P\{\mathfrak{N}_{\mathfrak{d}}(t)=k\}=\frac{1}{E_{\mathfrak{d}-1,\mathfrak{d}-1}\left(\left(\lambda t\right)^{\mathfrak{d}-1}\right)}
     \frac{(\lambda t)^{k(\mathfrak{d}-1)}}{\Gamma((k+1)(\mathfrak{d}-1))}\\
     \nonumber &=\frac{1}{E_{\frac{\mathfrak{d}-1}{2},\frac{\mathfrak{d}}{2},\frac{\mathfrak{d}-1}{2},
     \frac{\mathfrak{d}-1}{2}}\left(\left(\frac{\lambda t}{2}\right)^{\mathfrak{d}-1}\right)}
     \left(\frac{\lambda t}{2}\right)^{k(\mathfrak{d}-1)}
     \frac{1}{\Gamma(\frac{k+1}{2}(\mathfrak{d}-1)+\frac{1}{2})\Gamma((\frac{\mathfrak{d}-1}{2})(k+1))},
     \end{align}
     with $\lambda > 0$, $\mathfrak{d}\geq 2$, $k=0, 1,\dots,$ and
     \begin{equation}\nonumber
    E_{\frac{\mathfrak{d}-1}{2},\frac{\mathfrak{d}}{2},\frac{\mathfrak{d}-1}{2},
     \frac{\mathfrak{d}-1}{2}}\left(\left(\frac{\lambda
     t}{2}\right)^{\mathfrak{d}-1}\right)=\sum_{k=0}^{\infty}\left(\frac{\lambda t}{2}\right)^{k(\mathfrak{d}-1)}
     \frac{1}{\Gamma(\frac{k+1}{2}(\mathfrak{d}-1)+\frac{1}{2})\Gamma((\frac{\mathfrak{d}-1}{2})(k+1))}.
     \end{equation}
    is the multi-index Mittag-Leffler function.
    The randomization of $\mathfrak{N}_{\mathfrak{d}}(t)$, applied here, is different from that considered in the paper
     by \citet{ale}. This different randomization permits us to arrive at the PDE's governing the distribution
     of $\mathbf{X}(t)$.
    The distribution \eqref{b2} generalizes the distribution of the
    homogeneous Poisson process which is retrieved as a special case
    for $d=2$. The probability generating
    function of the generalized Poisson process
    $\mathfrak{N}_{\mathfrak{d}}(t)$, $t\geq 0$, is given by
    \begin{equation}
    G_{\mathfrak{d}}(u,t)= \frac{E_{\mathfrak{d}-1,\mathfrak{d}-1}\left(\left(\lambda
     t\right)^{\mathfrak{d}-1}u\right)}{E_{\mathfrak{d}-1, \mathfrak{d}-1}\left(\left(\lambda t\right)^{\mathfrak{d}-1}\right)}.
    \end{equation}
    It is simple to prove that the function
    \begin{equation}\label{genio}
    f(u,t)= u^{\mathfrak{d}-2}G_{\mathfrak{d}}(u^{\mathfrak{d}-1},t),
    \end{equation}
    satisfies the ordinary differential equation of order
    $\mathfrak{d}-1$
    \begin{equation}
    \frac{d^{\mathfrak{d}-1}f}{du^{\mathfrak{d}-1}}(u,t)=(\lambda
    t)^{\mathfrak{d}-1}f(u,t), \quad \mathfrak{d}\geq 2.
    \end{equation}
    In the special case $\mathfrak{d}=2$, the function \eqref{genio} coincides with the probability generating function
    of the homogeneous Poisson process.\\

     By combining \eqref{a2} and \eqref{b2}, we obtain the probability law
     \begin{align}\label{not}
      &P\{\mathbf{X}_{\mathfrak{d}}(t)\in d\mathbf{x}_\mathfrak{d}\}=
      \sum_{k=1}^{\infty}P\{\mathbf{X}_{\mathfrak{d}}(t)\in d\mathbf{x}_{\mathfrak{d}}|\mathfrak{N}_{\mathfrak{d}}(t)=k\}P\{\mathfrak{N}_{\mathfrak{d}}(t)=k\}\\
    \nonumber & =\frac{d\mathbf{x}_{\mathfrak{d}}}{\pi^{\mathfrak{d}/2}(ct)^{\mathfrak{d}-2}
       E_{\frac{\mathfrak{d}-1}{2},\frac{\mathfrak{d}}{2},\frac{\mathfrak{d}-1}{2},
     \frac{\mathfrak{d}-1}{2}}\left(\left(\frac{\lambda t}{2}\right)^{\mathfrak{d}-1}\right)}
      \times\sum_{k=1}^{\infty}\left(\frac{\lambda}{2c}\right)^{k(\mathfrak{d}-1)}
      \frac{(c^2t^2-\|\mathbf{x}_\mathfrak{d}\|^2)^{\frac{k}{2}(\mathfrak{d}-1)-1}}{\Gamma(\frac{k}{2}(\mathfrak{d}-1))
      \Gamma((\frac{\mathfrak{d}-1}{2})(k+1))}.
     \end{align}
     We remark that for $\mathfrak{d}=2$, we have that
    \begin{equation}\label{pt}
    \frac{P\{\mathbf{X}_2(t)\in d\mathbf{x}_2\}}{\prod_{j=1}^2 dx_j}=
    \frac{\lambda}{2\pi c}\frac{e^{-\lambda t+\frac{\lambda}{c}\sqrt{c^2t^2-x_1^2-x_2^2}}}{\sqrt{c^2t^2-x_1^2-x_2^2}} ,
    \end{equation}
    for $x_1^2+x_2^2<c^2t^2$; which is the absolutely continuous component of the distribution of a point
    $(X_1(t),X_2(t))$ performing the planar random motion studied in \citet{kol}.
    In this paper the authors proved
    that the density \eqref{pt} is the fundamental solution to the planar telegraph equation (also equation of damped waves)
    \begin{equation}
    \frac{\partial^2 u}{\partial t^2}+2\lambda\frac{\partial u}{\partial t}=
    c^2 \left\{\frac{\partial^2}{\partial x_1^2}+\frac{\partial^2}{\partial x_2^2}\right\}u.
    \end{equation}

     We are now ready to state the following
     \begin{te}\label{th1}
     The function
     \begin{align}
     f(\mathbf{x},t)&=\pi^{\mathfrak{d}/2}(ct)^{\mathfrak{d}-2}
       E_{\frac{\mathfrak{d}-1}{2},\frac{\mathfrak{d}}{2},\frac{\mathfrak{d}-1}{2},
     \frac{\mathfrak{d}-1}{2}}\left(\left(\frac{\lambda t}{2}\right)^{\mathfrak{d}-1}\right)
       \frac{P\{\mathbf{X}_\mathfrak{d}(t)\in d\mathbf{x}_\mathfrak{d}\}}{\prod_{j=1}^\mathfrak{d} dx_j}\\
      \nonumber &=\sum_{k=1}^{\infty}\left(\frac{\lambda}{2c}\right)^{k(\mathfrak{d}-1)}
      \frac{(c^2t^2-\|\mathbf{x}_\mathfrak{d}\|^2)^{\frac{k}{2}(\mathfrak{d}-1)-1}}{\Gamma(\frac{k}{2}(\mathfrak{d}-1))
      \Gamma((\frac{\mathfrak{d}-1}{2})(k+1))},\quad \mathbf{x}\in \mathbb{R}^{\mathfrak{d}}, t\geq 0,
     \end{align}
     solves the $\mathfrak{d}$-dimensional higher order Klein-Gordon equation
     \begin{equation}\label{cdim}
     \left(\frac{\partial^2}{\partial t^2}-c^2\Delta\right)^{\mathfrak{d}-1}u(\mathbf{x},t)= \lambda^{2(\mathfrak{d}-1)}
      u(\mathbf{x},t),
     \end{equation}
     where $\Delta = \sum_{j=1}^\mathfrak{d} \frac{\partial^2}{\partial x_j^2}$, $\mathfrak{d}\geq 2$.
     \end{te}

     \begin{proof}

     To begin with, we observe that by means of the transformation
        \begin{align*}
            w =\left(c^2t^2-\|\mathbf{x}_{\mathfrak{d}}\|^2 \right)^{1/2},
        \end{align*}
        applied in \eqref{cdim} we obtain
        \begin{equation}
            \label{Lddin}
            \left(\frac{d^2}{dw^2}+\frac{\mathfrak{d}}{w}\frac{d}{dw}\right)^{\mathfrak{d}-1}u(w)
            =\left(\frac{\lambda^2}{c^{2}}\right)^{\mathfrak{d}-1}u(w).
        \end{equation}
        The operator appearing in \eqref{Lddin} can be considered again as a
        particular case of the operator \eqref{L} with $a_1=-\mathfrak{d}$, $a_2=\mathfrak{d}$,
        $a_3=0$, $a=0$, $n=m=2$, $b_1=\frac{\mathfrak{d}-1}{2}$ and $b_2=0$. Hence, from Lemma \ref{pot} we have
        that
        \begin{equation}
            \left(\frac{d^2}{dw^2}+\frac{\mathfrak{d}}{w}\frac{d}{dw}\right)^{\mathfrak{d}-1}u (w)=
            L^{\mathfrak{d}-1}u(w)=4^{\mathfrak{d}-1}w^{-2(\mathfrak{d}-1)}
            I_2^{0,1-\mathfrak{d}}I_2^{\frac{\mathfrak{d}-1}{2}, 1-\mathfrak{d}}u (w).
        \end{equation}
       In view of Lemma \ref{brunello} we observe that
     \begin{equation}
      L^{\mathfrak{d}-1}w^{\beta}= 4^{\mathfrak{d}-1}w^{\beta-2(\mathfrak{d}-1)}
      \frac{\Gamma(\frac{\beta}{2}+1+\frac{\mathfrak{d}-1}{2})\Gamma(\frac{\beta}{2}+1)}{
       \Gamma(\frac{\mathfrak{d}-1}{2}+\frac{\beta}{2}+2-\mathfrak{d})\Gamma(\frac{\beta}{2}+2-\mathfrak{d})}.
     \end{equation}
     We now consider the function $f(\mathbf{x},t)$ in the new variable $w$
     \begin{equation}\nonumber
     f(w)=\sum_{k=1}^{\infty}\left(\frac{\lambda}{2c}\right)^{k(\mathfrak{d}-1)}\frac{w^{k(\mathfrak{d}-1)-2}
     }{\Gamma(k(\frac{\mathfrak{d}-1}{2}))
      \Gamma(\frac{\mathfrak{d}-1}{2}(k+1))}.
    \end{equation}
    From the previous calculations we have that
    \begin{align}
     L^{\mathfrak{d}-1}f(w)&=4^{\mathfrak{d}-1}\sum_{k=1}^{\infty}\left(\frac{\lambda}{2c}\right)^{k(\mathfrak{d}-1)}
     \frac{w^{k(\mathfrak{d}-1)-2-2(\mathfrak{d}-1)}}{\Gamma(\frac{\mathfrak{d}-1}{2}k+1-\mathfrak{d}))
      \Gamma(\frac{\mathfrak{d}-1}{2}k+1-\mathfrak{d}+\frac{\mathfrak{d}-1}{2})}\\
    \nonumber
    &=4^{\mathfrak{d}-1}\sum_{k'=-1}^{\infty}\left(\frac{\lambda}{2c}\right)^{(k+2)(\mathfrak{d}-1)}
        \frac{w^{k'(\mathfrak{d}-1)-2}}{\Gamma(k'(\frac{\mathfrak{d}-1}{2}))
      \Gamma(\frac{\mathfrak{d}-1}{2}(k'+1))}\\
    \nonumber &=\left(\frac{\lambda}{c}\right)^{2(\mathfrak{d}-1)} \sum_{k=1}^{\infty}\left(\frac{\lambda}{2c}\right)^{k(\mathfrak{d}-1)}
      \frac{w^{k(\mathfrak{d}-1)-2}}{\Gamma(\frac{k}{2}(\mathfrak{d}-1))
      \Gamma((\frac{\mathfrak{d}-1}{2})(k+1))}\\
    \nonumber
    &=\left(\frac{\lambda}{c}\right)^{2(\mathfrak{d}-1)}f(w),
    \end{align}
    where $k'=k-2$.
     This means that $f(w)$ satisfies the following equation
     \begin{equation}
        \left(\frac{d^2}{dw^2}+\frac{\mathfrak{d}}{w}\frac{d}{dw}\right)^{\mathfrak{d}-1}u(w)
            =\left(\frac{\lambda^2}{c^{2}}\right)^{\mathfrak{d}-1}u(w).
     \end{equation}
     By returning to the variables $(\mathbf{x},t)$, we finally obtain the claimed result.

     \end{proof}

     We now concentrate our attention to random flights in $\mathbb{R}^3$, which is clearly relevant for applications.
     From \eqref{not} we have that the absolutely continuous component of the probability law is given by
     \begin{align}
      \frac{P\{\mathbf{X}_{3}(t)\in d\mathbf{x}_3\}}{\prod_{j=1}^{3} dx_j}=
      \left(\frac{\lambda}{2c}\right)^2\frac{1}{\pi \sinh(\lambda t)}
     \frac{I_1\left(\frac{\lambda}{c}\sqrt{c^2t^2-\|\mathbf{x}_3\|^2}\right)}{\sqrt{c^2t^2-\|\mathbf{x}_3\|^2}}=p(\|\mathbf{x}_3\|,t).
     \end{align}
     Since
     \begin{equation}\label{sph}
     \int_{S^3_{ct}} P\{\mathbf{X}_{3}(t)\in d\mathbf{x}_3\}=1-\frac{\lambda t}{\sinh(\lambda t)}=
      1-P\{\mathfrak{N}_3(t)=0\},
     \end{equation}
     the distribution of $\mathbf{X}_3(t)$ has a singular component uniformly distributed on $\partial
     S^3_{ct}$, because the particle has initial uniformly
     distributed orientation.

     \begin{te}
     The probability law $p(\mathbf{x},t)=\frac{P\{\mathbf{X}_3(t)\in d\mathbf{x}_3\}}{\prod_{j=1}^3 dx_j}$ of the random flight in $\mathbb{R}^3$,
     is governed by the fourth-order, homogeneous partial differential equation with time-varying coefficients
     \begin{align}
     &\left(\frac{\partial^2}{\partial t^2}-c^2\Delta\right)^2 p(\mathbf{x},t)+2\lambda\left(\frac{\partial^2}{\partial t^2}-c^2\Delta\right)
      \left(\lambda+2b(t)\frac{\partial}{\partial t}\right)p(\mathbf{x},t)\\
     \nonumber &+ 4\lambda^2\left(\frac{\partial^2}{\partial t^2}+\lambda^2 b(t)\frac{\partial}{\partial t}\right)p(\mathbf{x},t)=0,
     \end{align}
     where $\mathbf{x}\in \mathbb{R}^3$ and
     \begin{align}
     \nonumber &b(t)=\coth(\lambda t).
     \end{align}

    \end{te}

    \begin{proof}
    From Theorem (3.1), we have that the function
    \begin{equation}\label{3dio}
    f(\mathbf{x},t)=\left(\frac{2c}{\lambda}\right)^2\pi
       \sinh(\lambda t)
       \frac{P\{\mathbf{X}_3(t)\in d\mathbf{x}_3\}}{\prod_{j=1}^3
       dx_j},
    \end{equation}
    satisfies the higher order 3-$\mathfrak{d}$ Klein-Gordon-type equation
    \begin{equation}\label{cadd}
    \left(\frac{\partial^2}{\partial t^2}-c^2\Delta\right)^2u(\mathbf{x},t)= \lambda^{4}
      u(\mathbf{x},t)
    \end{equation}
    Substituting the function \eqref{3dio} to \eqref{cadd}, we obtain the governing equation for the probability law of
    the random flight.

    \end{proof}

    We now consider the distribution $q_1(\mathbf{x}_2,t)=q_1(x_1,x_2,t)$ of the projection of the absolutely
    continuous component of the probability law of $\mathbf{X}_3(t)$ onto to the plane
    \begin{align}
    q_1(x_1,x_2,t)&=\int_{-\sqrt{c^2t^2-\|\mathbf{x}_2\|^2}}^{\sqrt{c^2t^2-\|\mathbf{x}_2\|^2}}
     p(\|\mathbf{x}_3\|,t)dx_3\\
    \nonumber &=\int_{-\sqrt{c^2t^2-\|\mathbf{x}_2\|^2}}^{\sqrt{c^2t^2-\|\mathbf{x}_2\|^2}}
     \left(\frac{\lambda}{2c}\right)^2\frac{1}{\pi \sinh(\lambda t)}
     \frac{I_1\left(\frac{\lambda}{c}\sqrt{c^2t^2-\|\mathbf{x}_3\|^2}\right)}{\sqrt{c^2t^2-\|\mathbf{x}_3\|^2}}dx_3\\
    \nonumber &=\left(x_3=\sqrt{w}\sqrt{c^2t^2-x_1^2-x_2^2}\right)\\
    \nonumber &=\left(\frac{\lambda}{2c}\right)^2\frac{1}{\pi \sinh(\lambda t)}\sum_{k=0}^{\infty}\left(\frac{\lambda}{2c}\right)^{2k+1}
    \frac{\left(\sqrt{c^2t^2-x_1^2-x_2^2}\right)^{2k+1}}{k!(k+1)!}\int_0^1\left(\sqrt{1-w^2}\right)^{2k}
    w^{-1/2}dw\\
    \nonumber &=\frac{\lambda}{2\pi c}\frac{1}{\sinh(\lambda t)}\frac{1}{\sqrt{c^2t^2-x_1^2-x_2^2}}
    \left[\cosh\left(\frac{\lambda}{c}\sqrt{c^2t^2-
     x_1^2-x_2^2}\right)-1\right],
    \end{align}
    for $x_1^2+x_2^2\leq c^2t^2$.
    Since the projection of the singular component of the
    distribution of $\mathbf{X}_3(t)$ onto the plane $(x_1,x_2)$ is
    equal to
    \begin{equation}
    q_2(x_1,x_2,t)=\frac{\lambda}{2\pi c}\frac{1}{\sinh(\lambda
    t)}\frac{1}{\sqrt{c^2t^2-x_1^2-x_2^2}},
    \end{equation}
    we have that the projection of the distribution of  $\mathbf{X}_3(t)$
    is given by
    \begin{align}\label{324}
    p(x_1,x_2,t)&=q_1(x_1, x_2,t)+ q_2(x_1,x_2,t)\\
    \nonumber &= \frac{\lambda}{2\pi c}\frac{1}{\sinh(\lambda t)}\frac{\cosh\left(\frac{\lambda}{c}\sqrt{c^2t^2-
     x_1^2-x_2^2}\right)}{\sqrt{c^2t^2-x_1^2-x_2^2}}
    \end{align}
    We can also consider the projection $X_1(t)$ of the distribution
    $\mathbf{X}_3(t)$ on the line. The distribution of $X_1(t)$ has
    a fine form and reads
    \begin{equation}\label{pro1}
    p(x_1,t)=\frac{\lambda I_0\left(\frac{\lambda}{c}\sqrt{c^2t^2-x_1^2}\right)}{2c\sinh(\lambda
    t)}, \quad |x|<ct.
    \end{equation}
    Furthermore \eqref{pro1} is a solution to the telegraph-type equation
    \begin{equation}
    \frac{\partial^2 p}{\partial t^2}+2\lambda \coth(\lambda t)\frac{\partial p}{\partial
    t}=c^2\frac{\partial^2 p}{\partial
    x^2}.
    \end{equation}
    This can be checked by considering that
    \begin{equation}\label{simil}
    \frac{2c}{\lambda}p(x_1,t)\cdot \sinh(\lambda
    t)=I_0\left(\frac{\lambda}{c}\sqrt{c^2t^2-x_1^2}\right),
    \end{equation}
    and the function
    \begin{equation}\nonumber
    q(x_1,t)=I_0\left(\frac{\lambda}{c}\sqrt{c^2t^2-x_1^2}\right),
    \end{equation}
    solves the equation
    \begin{equation}
    \frac{\partial^2q}{\partial t^2}-\lambda^2 q=c^2\frac{\partial^2 q}{\partial x^2}.
    \end{equation}
    In the same way it is simple to prove that the distribution \eqref{324} solves the two-dimensional
    telegraph-type equation
    \begin{equation}
    \left(\frac{\partial^2}{\partial t^2}+2\lambda \coth(\lambda t)\frac{\partial}{\partial
    t}-c^2\Delta\right)p(x_1,x_2,t)=0.
    \end{equation}
    We observe that the distribution of the random flight
    $\mathbf{X}_3(t)$ satisfies a fourth-order p.d.e. while its
    projections on the plane and on the line are directed by
    second-order p.d.e.'s of the telegraph form with one time-varying
    coefficient.

    \subsection{The second case}
     Let us consider the random flights in $\mathbb{R}^{\mathfrak{d}}$, $\mathfrak{d}\geq 3$, with intermediate step
     lengths having Dirichlet joint distribution with parameters
     $(\frac{\mathfrak{d}}{2}-1, \dots, \frac{\mathfrak{d}}{2}-1)$,
     that is
     \begin{equation}
     f_2(\tau_1, \dots, \tau_k)=\frac{\Gamma((k+1)(\frac{\mathfrak{d}}{2}-1))}{\Gamma(\frac{\mathfrak{d}}{2}-1)^{k+1}}
     \frac{1}{t^{(k+1)(\frac{\mathfrak{d}}{2}-1)-1}}\prod_{j=1}^{k+1}\tau_j^{\frac{\mathfrak{d}}{2}-2},
     \end{equation}
     where $0<\tau_j< t-\sum_{n=0}^{j-1}\tau_n$, $1\leq j\leq k$,
     $\tau_{k+1}=t-\sum_{j=1}^{k}\tau_j$.
     This kind of random flights were considered in \citet{Lec1} and \citet{ale}, where it was shown
     (Theorem 2) that the corresponding distribution of the moving point $\mathbf{Y}_{\mathfrak{d}}(t)=(Y_1(t), \dots, Y_{\mathfrak{d}}(t))$ is given by
     \begin{equation}\label{a}
     p_{\mathbf{Y}_\mathfrak{d}}(\mathbf{y}_\mathfrak{d}, t ; k)=\frac{\Gamma((k+1)(\frac{\mathfrak{d}}{2}-1)+1)}{\Gamma(k(\frac{\mathfrak{d}}{2}-1))}
     \frac{(c^2t^2-\|\mathbf{y}_{\mathfrak{d}}\|^2)^{k(\frac{\mathfrak{d}}{2}-1)-1}}{\pi^{\mathfrak{d}/2}(ct)^{2(k+1)(\frac{\mathfrak{d}}{2}-1)}},
     \end{equation}
     with $\|\mathbf{y}_{\mathfrak{d}}\|<ct$. In order to obtain the unconditional distributions,
     we assume here that the random number of changes of direction is
     endowed with the following distribution (depending on the
     dimension $\mathfrak{d}$ of the space)
     \begin{align}\label{b}
     &P\{\mathcal{N}_{\mathfrak{d}}(t)=k\}=\frac{1}{E_{\mathfrak{d}-2,\mathfrak{d}-1}
     \left((\lambda t)^{\mathfrak{d}-2}\right)}
     \frac{(\lambda t)^{k(\mathfrak{d}-2)}}{\Gamma((\mathfrak{d}-2)k+\mathfrak{d}-1)}\\
     \nonumber &=\frac{1}{E_{\frac{\mathfrak{d}}{2}-1,\frac{\mathfrak{d}}{2},\frac{\mathfrak{d}}{2}-1,\frac{\mathfrak{d}-1}{2}}
     \left(\left(\frac{\lambda t}{2}\right)^{\mathfrak{d}-2}\right)}
     \left(\frac{\lambda t}{2}\right)^{k(\mathfrak{d}-2)}
     \frac{1}{\Gamma((\frac{\mathfrak{d}}{2}-1)k+\frac{\mathfrak{d}}{2})\Gamma(\frac{\mathfrak{d}-1}{2}+(\frac{\mathfrak{d}}{2}-1)k)},
     \end{align}
     with $\lambda > 0$, $\mathfrak{d}\geq 3$, $k=0, 1,\dots,$ and
     where
     \begin{equation}\nonumber
      E_{\frac{\mathfrak{d}}{2}-1,\frac{\mathfrak{d}}{2},\frac{\mathfrak{d}}{2}-1,\frac{\mathfrak{d}-1}{2}}
     \left(\left(\frac{\lambda t}{2}\right)^{\mathfrak{d}-2}\right)=
      \sum_{k=0}^{\infty} \left(\frac{\lambda t}{2}\right)^{k(\mathfrak{d}-2)}
     \frac{1}{\Gamma((k+1)(\frac{\mathfrak{d}}{2}-1)+1)\Gamma(\frac{\mathfrak{d}-1}{2}+(\frac{\mathfrak{d}}{2}-1)k)}
     \end{equation}

     is the multi-index Mittag-Leffler function (see for example \citet{ky} and references
     therein). By combining \eqref{a} and \eqref{b}, we obtain the probability law
     \begin{align}\label{note}
      \frac{P\{\mathbf{Y}_{\mathfrak{d}}(t)\in d\mathbf{y}_{\mathfrak{d}}\}}{\prod_{j=1}^{\mathfrak{d}} dy_j}&= \frac{1}{\pi^{\mathfrak{d}/2}(ct)^{\mathfrak{d}-2}
       E_{\frac{\mathfrak{d}}{2}-1,\frac{\mathfrak{d}}{2},\frac{\mathfrak{d}}{2}-1,\frac{\mathfrak{d}-1}{2}}
     \left(\left(\frac{\lambda t}{2}\right)^{\mathfrak{d}-2}\right)}\times\\
      \nonumber &\times\sum_{k=1}^{\infty}\left(\frac{\lambda}{2c}\right)^{k(\mathfrak{d}-2)}
      \frac{(c^2t^2-\|\mathbf{y}_\mathfrak{d}\|^2)^{k(\frac{\mathfrak{d}}{2}-1)-1}}{\Gamma(k(\frac{\mathfrak{d}}{2}-1))
      \Gamma(\frac{\mathfrak{d}-1}{2}+(\frac{\mathfrak{d}}{2}-1)k)}.
     \end{align}
     We are now ready to state the following

     \begin{te}\label{primo}
     The function
     \begin{align}
     f(\mathbf{y},t)&=\pi^{\mathfrak{d}/2}(ct)^{\mathfrak{d}-2}E_{\frac{\mathfrak{d}}{2}-1,\frac{\mathfrak{d}}{2},\frac{\mathfrak{d}}{2}-1,\frac{\mathfrak{d}-1}{2}}
     \left(\left(\frac{\lambda t}{2}\right)^{\mathfrak{d}-2}\right)
       \frac{P\{\mathbf{Y}_\mathfrak{d}(t)\in d\mathbf{y}_\mathfrak{d}\}}{\prod_{j=1}^{\mathfrak{d}}
       dy_j}\\
    \nonumber &=\sum_{k=1}^{\infty}\left(\frac{\lambda}{2c}\right)^{k(\mathfrak{d}-2)}
      \frac{\left(\sqrt{c^2t^2-\|\mathbf{y}_\mathfrak{d}\|^2}\right)^{k(\mathfrak{d}-2)-2}}{\Gamma(k(\frac{\mathfrak{d}}{2}-1))
      \Gamma(\frac{\mathfrak{d}-1}{2}+(\frac{\mathfrak{d}}{2}-1)k)},
     \end{align}
     solves the $\mathfrak{d}-dimensional$ higher order non-homogeneous Klein-Gordon equation
     \begin{equation}\label{ddim}
     \left(\frac{\partial^2}{\partial t^2}-c^2\Delta\right)^{\mathfrak{d}-2}u(\mathbf{y},t)= \lambda^{2(\mathfrak{d}-2)}
      u(\mathbf{y},t)+\left(2\lambda c\right)^{\mathfrak{d}-2}
     \frac{(c^2t^2-\|\mathbf{y}_{\mathfrak{d}}\|^2)^{-\mathfrak{d}/2}}{\sqrt{\pi}\Gamma(1-\frac{\mathfrak{d}}{2})},
     \end{equation}
     where $\Delta = \sum_{j=1}^{\mathfrak{d}} \frac{\partial^2}{\partial y_j^2}$.
     \end{te}

     \begin{proof}

    The proof follows the same reasoning used in Theorem \ref{th1}.\\
     To begin with, we observe that by means of the transformation
        \begin{align*}
            w =\left(c^2t^2-\|\mathbf{y}_{\mathfrak{d}}\|^2 \right)^{1/2},
        \end{align*}
        we convert \eqref{ddim} into
        \begin{equation}
            \label{Lddi}
            \left(\frac{d^2}{dw^2}+\frac{\mathfrak{d}}{w}\frac{d}{dw}\right)^{\mathfrak{d}-2}u(w)
            =\left(\frac{\lambda^2}{c^{2}}\right)^{\mathfrak{d}-2}u(w)+
            \left(\frac{2\lambda}{c}\right)^{\mathfrak{d}-2}\frac{w^{-\mathfrak{d}}}{\sqrt{\pi}
            \Gamma(1-\frac{\mathfrak{d}}{2})}.
        \end{equation}
        The operator appearing in \eqref{Lddi} can be considered again as a
        specific case of the operator \eqref{L} with $a_1=-\mathfrak{d}$, $a_2=\mathfrak{d}$,
        $a_3=0$, $a=0$, $n=m=2$, $b_1=\frac{\mathfrak{d}-1}{2}$ and $b_2=0$. Hence, from Lemma \ref{pot} we have
        that
        \begin{equation}
            \left(\frac{d^2}{dw^2}+\frac{\mathfrak{d}}{w}\frac{d}{dw}\right)^{\mathfrak{d}-2}u (w)
            =L^{\mathfrak{d}-2}u(w)=4^{\mathfrak{d}-2}w^{-2(\mathfrak{d}-2)}
            I_2^{0,2-\mathfrak{d}}I_2^{\frac{\mathfrak{d}-1}{2}, 2-\mathfrak{d}}u (w).
        \end{equation}
       In view of Lemma \ref{brunello} we observe that
     \begin{equation}
      L^{\mathfrak{d}-2}w^{\beta}= 4^{\mathfrak{d}-2}w^{\beta-2(\mathfrak{d}-2)}
      \frac{\Gamma(\frac{\beta}{2}+1+\frac{\mathfrak{d}-1}{2})
      \Gamma(\frac{\beta}{2}+1)}{
       \Gamma(\frac{\mathfrak{d}-1}{2}+1+\frac{\beta}{2}+2-\mathfrak{d})\Gamma(\frac{\beta}{2}-\mathfrak{d}+3)}.
     \end{equation}
     We now take into account the function $f(\mathbf{y},t)$ in the new variable $w$
     \begin{equation}\nonumber
     f(w)=\sum_{k=1}^{\infty}\left(\frac{\lambda}{2c}\right)^{k(\mathfrak{d}-2)}
     \frac{w^{k(\mathfrak{d}-2)-2}}{\Gamma(k(\frac{\mathfrak{d}}{2}-1))
      \Gamma(\frac{\mathfrak{d}-1}{2}+(\frac{\mathfrak{d}}{2}-1)k)}
    \end{equation}
    From the previous calculations we have that
    \begin{align}
     L^{\mathfrak{d}-2}f(w)&=4^{\mathfrak{d}-2}\sum_{k=1}^{\infty}\left(\frac{\lambda}{2c}\right)^{k(\mathfrak{d}-2)}
     \frac{w^{k(\mathfrak{d}-2)-2-2(\mathfrak{d}-2)}}{\Gamma(\frac{\mathfrak{d}-2}{2}k-\mathfrak{d}+2))
      \Gamma(\frac{\mathfrak{d}-2}{2}k+\frac{\mathfrak{d}-1}{2}+2-\mathfrak{d})}\\
    \nonumber
    &=4^{\mathfrak{d}-2}\sum_{k'=-1}^{\infty}\left(\frac{\lambda}{2c}\right)^{(k+2)(\mathfrak{d}-2)}
        \frac{w^{k'(\mathfrak{d}-2)-2}}{\Gamma(k'(\frac{\mathfrak{d}}{2}-1))
      \Gamma(\frac{\mathfrak{d}-1}{2}+(\frac{\mathfrak{d}}{2}-1)k')}\\
    \nonumber
    &=\left(\frac{\lambda}{c}\right)^{2(\mathfrak{d}-2)}\left[f(w)+\left(\frac{\lambda}{2c}\right)^{2-\mathfrak{d}}
                     \frac{w^{-\mathfrak{d}}}{\sqrt{\pi}\Gamma(1-\frac{\mathfrak{d}}{2})}\right],
    \end{align}
    where $k'=k-2$.
     This means that $f(w)$ satisfies the following equation
     \begin{equation}
        \left(\frac{d^2}{dw^2}+\frac{\mathfrak{d}}{w}\frac{d}{dw}\right)^{\mathfrak{d}-2}u(w)
            =\left(\frac{\lambda^2}{c^{2}}\right)^{\mathfrak{d}-2}u(w)+\left(\frac{2\lambda}{c}\right)^{\mathfrak{d}-2}
            \frac{w^{-\mathfrak{d}}}{\sqrt{\pi}\Gamma(1-\frac{\mathfrak{d}}{2})}.
     \end{equation}
     By going back to the variables $(\mathbf{y},t)$, we arrive at the claimed result.
     \end{proof}

     \begin{os}
     We observe that the inhomogeneous term in \eqref{ddim} vanishes
     for all even values of $\mathfrak{d}\geq 4$.
     \end{os}

    From \eqref{note} we have that the absolutely continuous component of the probability law of the random flight in $\mathbb{R}^3$ is given by
     \begin{align}\label{secpr}
      \frac{P\{\mathbf{Y}_{3}(t)\in d\mathbf{y}_3\}}{\prod_{j=1}^{3} dy_j}=
      \frac{\lambda}{2c}\frac{1}{\pi (e^{\lambda t}-1)}
     \sum_{k=0}^{\infty}\left(\frac{\lambda}{2c}\right)^{k+1}\frac{\left(\sqrt{c^2t^2-\|\mathbf{y}_3\|^2}\right)^{k-1}}{\Gamma(\frac{k+1}{2})\Gamma(\frac{k+3}{2})}.
     \end{align}
     The singular part of the distribution of $\mathbf{Y}_{3}(t)$ is uniform on
     $S_{ct}^3$ and has weight equal to
     \begin{equation}
     \int_{S^3_{ct}} P\{\mathbf{Y}_{3}(t)\in d\mathbf{y}_3\}=1-\frac{\lambda t}{e^{\lambda t}-1}=
      1-P\{\mathcal{N}_3(t)=0\}.
     \end{equation}

    \begin{te}
     The probability law of the random flight in $\mathbb{R}^3$
     is governed by the following non-homogeneous 3-$\mathfrak{d}$ telegraph equation with variable coefficients
     \begin{equation}\label{varte}
     \left(\frac{\partial^2}{\partial t^2}+c_1(t)\frac{\partial }{\partial t}-c^2 \Delta \right)u(\mathbf{y},t)=
     c_2(t) u(\mathbf{y},t)+c_3(\mathbf{y},t),
     \end{equation}
     where $\mathbf{y}\in \mathbb{R}^3$ and
     \begin{align}
     \nonumber &c_1(t)=\frac{2\lambda e^{\lambda t}}{e^{\lambda t}-1}\\
     \nonumber &c_2(t)=-\frac{\lambda^2}{e^{\lambda t}-1}\\
     \nonumber &c_3(\mathbf{y}, t)=\frac{\lambda^2}{\sqrt{\pi^3}(e^{\lambda t}-1)}
     \frac{(c^2t^2-\|\mathbf{y}_3\|^2)^{-3/2}}{\Gamma(-\frac{1}{2})}.
     \end{align}

    \end{te}

    \begin{proof}
    From the previous theorem, we have that the function
    \begin{equation}\label{3d}
    f(\mathbf{y},t)=2\pi ct
       \frac{e^{\lambda t}-1}{\lambda t}
       \frac{P\{\mathbf{Y}_3(t)\in d\mathbf{y}_3\}}{\prod_{j=1}^3
       dy_j},
    \end{equation}
    satisfies the inhomogeneous 3-$\mathfrak{d}$ Klein-Gordon equation
    \begin{equation}\label{3dd}
    \left(\frac{\partial^2}{\partial t^2}-c^2\Delta\right)u(\mathbf{y},t)= \lambda^{2}
      u(\mathbf{y},t)+\left(2\lambda c\right)
     \frac{(c^2t^2-\|\mathbf{y}_3\|^2)^{-3/2}}{\sqrt{\pi}\Gamma(-\frac{1}{2})}
    \end{equation}
    By substituting the function \eqref{3d} into \eqref{3dd}, we obtain the equation governing the probability law of
    the random flight $\mathbf{Y}_3(t)$.

    \end{proof}

    \begin{os}
    The law of the projection $(Y_1(t), Y_2(t))$ on the plane of the
    random motion considered here ($\mathbf{Y}_3(t)$) reads
    \begin{equation}\label{sepr}
    P\{Y_1(t)\in dy_1, Y_2(t)\in dy_2\}=\frac{\lambda dy_1 dy_2}{2\pi c(e^{\lambda t}-1)}
    \frac{e^{\frac{\lambda}{c}\sqrt{c^2t^2-(y_1^2+y_2^2)}}}{\sqrt{c^2t^2-(y_1^2+y_2^2)}},
    \end{equation}
    for $(y_1,y_2)\in C_{ct}$, where $C_{ct}=\{y_1,y_2:y_1^2+y_2^2\leq
    c^2t^2\}$.\\
    This can be checked by integrating \eqref{secpr} w.r.
    to $y_3$ and then by summing the contribution of the projection of
    the singular component of the distribution as performed in
    \eqref{324}.\\
    Furthermore, we observe that the two-dimensional distribution
    \eqref{sepr} satisfies the following time-varying telegraph-type
    equation
    \begin{equation}
    \left(\frac{\partial^2}{\partial t^2}-c^2 \Delta\right)p+
    c_1(t)\frac{\partial p}{\partial
    t}-c_2(t)p=0,
    \end{equation}
    where the functions $c_1(t)$ and $c_2(t)$ are defined in Theorem 3.5.
    The behavior of the distributions \eqref{324} and \eqref{sepr}
    near the edge of the circles $C_{ct}$ is similar. We note that
    \begin{equation}
    \lim_{y_1, y_2\rightarrow 0} \frac{P\{Y_1(t)\in dy_1, Y_2(t)\in dy_2\}}{dy_1dy_2}=\frac{\lambda}{2\pi c^2t}
    \frac{1}{1-e^{-\lambda t}},
    \end{equation}
    while
    \begin{equation}
    \lim_{x_1,x_2\rightarrow 0} \frac{P\{X_1(t)\in dx_1, X_2(t)\in dx_2\}}{dx_1dx_2}=\frac{\lambda}{2\pi c^2t}
    \frac{e^{\lambda t}+e^{-\lambda t}}{e^{\lambda t}-e^{-\lambda
    t}}.
    \end{equation}
    \end{os}

    \begin{os}
    From \eqref{sepr} we can infer that
    \begin{equation}
    P\{Y_1(t)\in dy_1\}=\frac{\lambda}{2c(e^{\lambda
    t}-1)}\sum_{k=0}^{\infty}\left(\frac{\lambda}{2c}\sqrt{c^2t^2-y_1^2}\right)^{k}
    \frac{1}{[\Gamma(\frac{k}{2}+1)]^2}.
    \end{equation}
    This distribution is similar to \eqref{simil} which was obtained
    as a projection of the planar random motion where changes of
    direction are paced by a homogeneous process. Both these
    probability distributions refer to one-dimensional random
    motions with random velocities (see also \citet{za} on this
    point).

    \end{os}

    \section{Three-dimensional random flights governed by a Poisson process}

    In \citet{ale} a random motion in $\mathbb{R}^3$ governed by a
    Poisson process was introduced. In more detail, the authors
    studied a random motion where particles change direction only at
    even-valued Poisson events.
    They show that, in this case, the unconditional probability law
    is given by
    \begin{equation}\label{p3}
    \frac{P\{\mathbf{U}_3(t)\in d\mathbf{u}_3,
    \cup_{k=1}^{\infty}(N(t)=2k+1)\}}{\prod_{j=1}^{3}du_j}=\frac{e^{-\lambda
    t}}{\pi}\left(\frac{\lambda}{2c}\right)^2\frac{1}{\sqrt{c^2t^2-\|\mathbf{u}_3\|^2}}
    I_1\left(\frac{\lambda}{c}\sqrt{c^2t^2-\|\mathbf{u}_3\|^2}\right).
    \end{equation}

    \begin{te}
    The probability law \eqref{p3} satisfies the 3-$\mathfrak{d}$
    telegraph equation
    \begin{equation}\label{xte}
    \left(\frac{\partial^2}{\partial t^2}+2\lambda \frac{\partial}{\partial
    t}-c^2\Delta\right)u(\mathbf{x},t)=0,
    \end{equation}
    where $\mathbf{x}\in \mathbb{R}^3$.
    \end{te}

    \begin{proof}
    By means of the exponential substitution
    \begin{equation}
    u(\mathbf{x},t)=e^{-\lambda t}f(\mathbf{x},t),
    \end{equation}
    \eqref{xte} reduces to 
    \begin{equation}\label{obe}
    \left(\frac{\partial^2}{\partial
    t^2}-c^2\Delta\right)f(\mathbf{x},t)=\lambda^2 f(\mathbf{x},t).
    \end{equation}
    By using the transformation
    \begin{equation}\nonumber
    w=\sqrt{c^2t^2-\|\mathbf{u}_3\|^2},
    \end{equation}
    we convert \eqref{obe} to the Bessel equation
    \begin{equation}
    \frac{d^2f}{dw^2}+\frac{1}{w}\frac{df}{dw}=\frac{\lambda^2}{c^2}f.
    \end{equation}
    We now observe that the function \eqref{p3}
    can be written as
    \begin{equation}
    \frac{P\{\mathbf{U}_3(t)\in d\mathbf{u}_3,
    \cup_{k=1}^{\infty}(N(t)=2k+1)\}}{\prod_{j=1}^{3}du_j}=e^{-\lambda t}f(\mathbf{x},t),
    \end{equation}
    where
    \begin{equation}
    f(\mathbf{x},t)=\frac{1}{\pi}\left(\frac{\lambda}{2c}\right)^2\frac{1}{\sqrt{c^2t^2-\|\mathbf{u}_3\|^2}}
    I_1\left(\frac{\lambda}{c}\sqrt{c^2t^2-\|\mathbf{u}_3\|^2}\right),
    \end{equation}
    solves \eqref{obe} and thus the proof of the theorem is complete.
    \end{proof}

\end{document}